\documentclass[11pt,righttag]{amsart}
\usepackage{diagrams}
\usepackage{graphicx}

\newtheorem{thm}{Theorem}[section]
\newtheorem{lemma}[thm]{Lemma}

\newtheorem{prop}[thm]{Proposition}

\theoremstyle{definition}
\newtheorem{defn}[thm]{Definition}

\newtheorem{rmk}[thm]{Remark}

\newcommand{\N}{{\mathbb N}}

\newcommand{\F}{{\mathbb F}}

\newcommand{\cM}{{\mathcal M}}

\newcommand{\Ext}{\hbox{{\rm Ext}}}
\newcommand{\Tor}{\hbox{{\rm Tor}}}

\newcommand{\mydot}{\hbox{$\,>\kern -11pt \cdot\kern 5pt$}}
\newcommand{\notmydot}{\hbox{$\,\not>\kern -11pt \cdot\kern 5pt$}}

\begin{document}

%%%%%%Title Page

\title[Splitting Algebras II]{Splitting Algebras II: The Cohomology Algebra}


\author[  Shelton ]{Brad Shelton}
\address{University of Oregon\\
Eugene, Oregon 97403}
\email{shelton@uoregon.edu}

\begin{abstract}
\baselineskip12pt
Gelfand, Retakh, Serconek and Wilson, in \cite{GRSW}, defined a graded algebra $A_\Gamma$ attached to any finite ranked poset 
$\Gamma$ - a generalization of the universal algebra of pseudo-roots of noncommutative polynomials.  This algebra has since come to 
be known as the {\it splitting algebra} of $\Gamma$.  The splitting algebra has a secondary filtration related to the rank function on the 
poset and the associated graded algebra is denoted here by $A'_\Gamma$.  We calculate the cohomology algebra (and coalgebra) of $
A'_\Gamma$ explicitly.  As a corollary to this calculation we have a proof that $A'_\Gamma$ is Koszul (respectively quadratic) if and only if 
$\Gamma$ is Cohen-Macaulay (respectively uniform).  We show by example that the cohomology algebra 
(resp. coalgebra) of $A_\Gamma$ may 
be strictly smaller that the cohomology algebra (resp. coalgebra) of $A'_\Gamma$.
\end{abstract}
\date{August 3, 2012}

\maketitle

%\bigskip

%%%%%%%%%%   Start of paper
\baselineskip18pt

\section{Introduction}

We fix a field $\F$.  All topological cohomology groups are calculated with coefficients in $\F$.  

Let $\Gamma$ be a finite ranked poset with unique minimal element $*$, strict order $<$ and rank function 
$rk(\cdot)$.  Write $x\to y$ if $x$ covers $y$ in the usual sense.  The {\it splitting algebra} of $\Gamma$, 
$A_\Gamma$, was introduced by Gelfand, Retakh, Serconek and Wilson in \cite{GRSW} and generalizes the universal algebra of pseudo-roots, 
$Q_n$, introduced in \cite{GRW}.  An explicit definition of $A_\Gamma$ is reproduced here in \ref{original}.   

The algebra $A_\Gamma$ has a natural filtration $F^pA_\Gamma$ induced by 
the rank function on $\Gamma$,  with associated graded  algebra denoted $A'_\Gamma$.  The orginal grading from 
$A_\Gamma$ is preserved and the filtration also induces  a filtration on $\Ext_{A_\Gamma}(\F,\F)$ as an algebra and 
a filtration on $\Tor^{A_\Gamma}(\F,\F)$ as a coalgebra. 

For any $b\in \Gamma$ and $1\le n\le rk(b)$ we let
$$\Gamma_{b,n} = \{a\in (*,b) \, |\, rk(b)-rk(a) \le n-1\}.$$
 For any finite poset $P$ we denote the {\it order complex} of $P$ by $\Delta(P)$. (The order complex is the simplicial complex of all strictly decreasing sequences in $P$.)    Our first main theorem is:

\begin{thm}\label{intro1}  For any finite ranked poset $\Gamma$ and $p\ge 1$,
$$\Ext_{A'_\Gamma}^{p.q}(\F,\F) \cong 
\bigoplus\limits_{*\ne b\in \Gamma \atop rk(b) \ge q} \tilde H^{p-2}(\Delta(\Gamma_{b,q})).$$
and
$$\Tor^{A'_\Gamma}_{p.q}(\F,\F) \cong 
\bigoplus\limits_{*\ne b\in \Gamma \atop rk(b) \ge q} \tilde H_{p-2}(\Delta(\Gamma_{b,q})).$$
\end{thm}

Theorem \ref{intro1} utilizes the standard conventions
$\tilde H^n(\Delta(\emptyset)) = 0$ for $n\ne -1$ and $\tilde H^{-1}(\Delta(\emptyset)) = \F$.  

The theorem was motivated by the calculation of the Hilbert series of $A_\Gamma$ (or equivalently $A'_\Gamma$), as given in 
\cite{RSW4}.  In turn, \ref{intro1} gives a new proof of that theorem.  Such a  statement is a bit misleading since the proof of \ref{intro1} 
relies heavily on the canonical basis result of \cite{GRSW}, which is also the essential fact used to prove the Hilbert series calculation.  
There is a significant difference, however:  in the calculation of the Hilbert series, the basis served as something to count, while in our 
work the basis serves the purpose of helping to define a homotopy on a specific chain complex.

Recall that a poset $\Gamma$ is {\it Cohen-Macaulay} if, for all $a<b$ in $\Gamma$:
$$\tilde H^n(\Delta((a,b)))\ne 0 \hbox{ implies } n = \dim(\Delta((a,b)).$$
Recall also from \cite{RSW2} the definition of {\it uniform} for $\Gamma$ (see \ref{uniform}).  
As an application of \ref{intro1} we easily obtain:

\begin{thm}\label{intro2}
For any finite ranked poset $\Gamma$:

{\rm (1)} $A'_\Gamma$ is quadratic if and only if $\Gamma$ is uniform.

{\rm (2)} $A'_\Gamma$ is Koszul if and only if $\Gamma$ is Cohen-Macaulay.

\end{thm}

The ``if'' part of \ref{intro2}, (1) goes all the way back to \cite{RSW2}, but the proof given here is new. 
The ``if'' part of \ref{intro2} (2) can be found in \cite{KS}.  Again, the proof here is entirely new.  Theorem \ref{intro2} places the algebras $A'_\Gamma$ onto the long list of classical and neo-classical results relating homological properties of algebras to the Cohen-Macaulay 
property of posets (see for example: \cite{BGSsurvey}, \cite{Polo}, \cite{Reisner}, \cite{Stanley75} and \cite{Woodcock}). 

Since $A'_\Gamma$ is an associated graded algebra of $A_\Gamma$, there is a standard spectral sequence with $E^1$-page 
$E^1= \Ext_{A'_\Gamma}(\F,\F)$, converging to $gr(\Ext_{A_\Gamma}(\F,\F))$.   In particular the dimension of the space 
$\Ext_{A'_\Gamma}^{p,q}(\F,\F)$ must 
dominate the dimension of the space $\Ext_{A_\Gamma}^{p,q}(\F,\F)$.  But there is no particular reason why these dimensions need to be the same and indeed we have:

\begin{thm}\label{intro3} There exist finite ranked posets $\Gamma$ for which:
$$gr(\Ext_{A_\Gamma}(\F,\F)) \ne  \Ext_{A'_\Gamma}(\F,\F)$$
and
$$gr(\Tor^{A_\Gamma}(\F,\F)) \ne \Tor^{A'_\Gamma}(\F,\F).$$
\end{thm}

We note the connection of this paper to the companion paper \cite{KS}.  In that paper it was shown that for a certain auxiliary algebra, 
$R_\Gamma$,  the following are equivalent: (1) $\Gamma$ is uniform and $R_\Gamma$ is Koszul, (2)  $\Gamma$ is Cohen-Macaulay.  When the poset $\Gamma$ is uniform, the algebra $A_\Gamma$ is quadratic and $R_\Gamma$ is simply the quadratic dual of the algebra 
$A'_\Gamma$.  Thus \ref{intro2} and \ref{intro3} combine to give a new proof of the result quoted from \cite{KS}.   It 
was also shown in \cite{KS} that when $\Gamma$ is Cohen-Macaulay, the cohomology of the order complex of 
$\Gamma\setminus\{*\}$ can be recovered from certain cohomology classes attached to $R_\Gamma$.  That 
result has no analog here. It was additionally shown in \cite{KS} that there are non-Koszul but numerically 
Koszul algebras of the form $R_\Gamma$ ($\Gamma$ uniform).  In light of \ref{intro3}, we still do not know if  
all such examples also generate non-Koszul but numerically Koszul algebras $A_\Gamma$ (although some such examples clearly do generate non-Koszul numerically Koszul algebras $A_\Gamma$). 

Finally, the results here shed light on many of the previous known results about splitting algebras.  In addition to the overlap with results from \cite{RSW2}, \cite{RSW4} and \cite{KS} already mentioned, many of the results of 
\cite{CPS}, \cite{RSW3} and \cite{SadSh}  follow immediately from the results here.

\section{Definitions, preliminaries and the canonical basis}\label{prelim}

\begin{defn} A finite ranked poset is a poset $\Gamma$, with strict order $<$, satisfying the following two properties:

(1) $\Gamma$ has unique minimal element $*$,

(2)  For any $x\in \Gamma$ any two maximal chains in $[*,x]$ have the same length. 
\end{defn}

The common maximal chain length in (2) above is the {\it rank} of $x$ in $\Gamma$ and is denoted $rk_\Gamma(x)$, 
or if there is no possiblity of confusion simply $rk(x)$. Whenever $y<x$ we write $d(x,y) = rk(x)-rk(y)$.  
Let $\Gamma(k)$ be the elements of $\Gamma$ of rank $k$. It is convenient to let $\Gamma_+ = \Gamma\setminus\{*\}$.  For 
$x,y\in \Gamma$ we say $x$ covers $y$, and write $x\to y$ if $y<x$ and $d(x,y)=1$.  This definition
makes  $\Gamma$ into a directed ({\it layered}) graph with edges $x\to y$ and layers $\Gamma(k)$.  We will typically write either $e=x\to y$ or $x\buildrel e\over \to y$ to indicate that $e$ is the directed edge from $x$ to $y$.

For any $a<b$ in $\Gamma$, let $\Pi(b,a)$ be the set of all paths 
$$\pi=(b\buildrel e_1\over \to b_1\buildrel e_2\over \to b_2 \buildrel e_3\over \to \cdots b_{n-1} 
	\buildrel e_n \over \to a)$$
where $n=d(b,a)$.   Let $W$ be the $\F$-vector space whose basis is the edges of the graph $\Gamma$ and 
$T_\F(W)$ the free $\F$-algebra on $W$. Let 
$s$ be a central indeterminate.  Given $\pi\in \Pi(b,a)$, as above, define:
$$P(\pi,s) = (s-e_1)(s-e_2)\cdots (s-e_n)   = \sum\limits_{j=0}^ n e(\pi,j) s^{n-j} \in T_\F(W)[s].$$
Note that $e(\pi,0)=1$ and $e(\pi,n) = e_1e_2\cdots e_n$. 

\begin{defn}\label{original}{\rm (\cite{GRSW})}  The splitting algebra of $\Gamma$ over the field $\F$ is the quotient algebra $A_\Gamma=T_\F(W)/I$ where $I$ is the ideal generated by 
$$\{e(\pi,j) - e(\pi',j) \,|\, \pi,\pi'\in \Pi(b,a), \, a<b\in \Gamma,\, 0\le j\le d(b,a)\}.$$
\end{defn}

We let each edge have degree 1 and note that the relation $e(\pi,j) = e(\pi',j)$ is homogeneous of degree $j$, so that $A_\Gamma$ inherits a grading from $T_\F(W)$. 

We note that many of the relations of $A_\Gamma$ are linear.  The following definitions allow one to simultaneously
 eliminate the linear relations and describe a canonical basis of $A_\Gamma$.   For each 
 $x\in \Gamma_+$, choose arbitrarily one edge $e_x$ of the form $x\to y$.  The choice is irrelevant, but 
 must be fixed.  We refer to these as distinguished edges. Then it is easy to see that the elements $e_x$
 form a linearly independent set of generators of 
$A_\Gamma$.   At the same time, for each $x\in \Gamma_+$ let $\pi_x$ be the unique path from 
$x$ to $*$ given by following the distinguished edges and set 
$v_x = - e(\pi_x,1)$.   Then the elements $v_x$ are also a linearly independent set of generators of $A_\Gamma$. 
For convenience we set $\pi_*$ to be the empty path and $v_* = 0$. Conveniently, 
if $e=x\to y$ is any edge, then $e = v_x - v_y$ in $A_\Gamma$.  

For each $x\in \Gamma_+$ and each $1\le j\le rk(x)$ we define $\hat e(x,j)$ to be the image 
in $A_\Gamma$ of the product $e_1e_2\cdots e_j$ where $\pi_x$ is the path 
$x \buildrel e_1\over \to x_1 \buildrel e_2\over \to \cdots \buildrel e_n\over \to *$.  Note that the 
elements $\hat e(x,j)$ are always products of distinguished edges. 

Finally, define a partial order on $\Gamma_+ \times \N$ by $(x,j)\mydot (y,k)$ if and only if $x>y$ and $j = d(x,y)$.  

\begin{defn}\label{good}  Consider a monomial $m$ of the form 
$\hat e(x_1,k_1)\hat e(x_2,k_2)\cdots \hat e(x_r,k_r)$, for some 
$r\ge 0$, $x_i\in \Gamma_+$, $1\le k_i\le rk(x_i)$.  We say that $m$ is a good monomial if for each $1\le i\le r-1$,  
$(x_i,k_i)\notmydot (x_{i+1},k_{i+1})$.    The set of all good monomials will be denoted $\cM_\Gamma$.
\end{defn}

\begin{thm}\label{basis}{\rm (\cite{GRSW}, \cite{RSW1})} Let $\Gamma$ be a finite ranked poset.  
Then the good monomials form a linear basis of $A_\Gamma$. 
\end{thm}

The rank function on $\Gamma$ induces a rank filtration on $A_\Gamma$, $F^nA_\Gamma$.  The rank of an 
edge $e=x\to y$ is defined to be the rank of $x$ and $F^nA_\Gamma$ is the span of all products of edges,
 $e_1e_2\cdots e_r$ for which $\sum rk(e_i) \le n$.   The associated graded algebra with respect to this filtration 
 will be denoted $A'_\Gamma$.

We note that $\hat e(x,j) \in F^{m(x,j)}A_\Gamma$, where $m(x,j):= j\cdot rk(x) - j(j-1)/2$.  It is convenient to notice that whenever 
$(x,j)\mydot (y,k)$, $m(x,j)+m(y,k) = m(x,j+k)$.

It is useful to have some notation for elements of the associated graded algebra $A'_\Gamma$.  
 
If $e=x\to y$ is an edge of $\Gamma$ and $b\in \Gamma_+$, then we write $e'$ and $v'_b$ for the associated
 elements of $A'_\Gamma$.  Note that in $A_\Gamma$, $e = v_x-v_y$ and hence in $A'_\Gamma$, 
 $e' = v_x'$.  Similarly, for $1\le j\le rk(x)$, we let $e'(x,j)$ be the image in $A'_\Gamma$ of the element $\hat e(x,j)$.  
 It is clear that we have the following $\F$-basis for $A'_\Gamma$:
$$\cM'_\Gamma = \{ e'(x_1,k_1)e'(x_2,k_2)\cdots e'(x_r,k_r)\,|\, r\ge 0, (x_i,k_i)\notmydot (x_{i+1},k_{i+1}) \hbox{ for } 1\le i< r\}.$$
We also refer to these as ``good'' monomials.  We conclude with one fundamental remark and one piece of notation.
\begin{rmk}\label{fundamental}
If $(x,j) \mydot (y,k)$ then $e'(x,j)e'(y,k) = e'(x,j+k)$.
\end{rmk}  

\begin{defn}  For any $b>a$ in $\Gamma$ and $k=d(b,a)$ we set $f(b,a) = e'(b,k)$ in $A'_\Gamma$.
\end{defn}

Equivalently, $f(b,a)$ is the consecutive product of any set of edges leading from $b$ to $a$. One should notice that $f(b,a)f(a,c) = f(b,c)$ whenever $b>a>c$.  It is also important to remember that $f(b,*) = e'(b,rk(b))$ is 
well-defined.  The reader is reminded to observe that the element $a$ cannot be retrieved from the notation $f(b,a)$ (but by the basis statement above, the element $b$ can be retrieved).  

\section{Proof of Theorem \ref{intro1}}

Let $\Gamma$ be a fixed finite ranked poset.  Throughout this section, let $A=A_\Gamma$ and $A'=A'_\Gamma$.  

Because we will be working with the reduced homology chain complexes of many different order complexes at the same time,  we 
must have notation to keep track of which space a given $n$-chain belongs to.  

\begin{defn} Given elements 
$b>b_n>b_{n-1}>\cdots >b_0$ in $\Gamma_+$, with $d(b,b_0)\le q-1$, we will write 
${(b_n>\cdots>b_0)\over b,q}$ to denote the corresponding $n$-chain basis element of the 
$\F$-space of $n$-chains $C_n(\Delta(\Gamma_{b,q}))$.  In particular the symbol ${(\,)\over b,q}$ stands for the basis element of 
$C_{-1}(\Delta(\Gamma_{b,q})) = \F$.  
\end{defn}

The efficiency of this unattractive notation will become apparent.  Next we need notation for the sum of all the 
appropriate order complex homologies.    

\begin{defn}  For each $b\in \Gamma_+$ and $1\le q\le rk(b)$, let $(C_\bullet(b,q), \delta_{b,q})$ be the canonical chain complex for reduced homology of $\Delta(\Gamma_{b,q})$:
$$  \cdots \to C_n(\Delta(\Gamma_{b,q})) \to \cdots  \to C_0(\Delta(\Gamma_{b,q})) \to C_{-1}(\Delta(\Gamma_{b,q})) \to 0.$$  

We define $C_\bullet(q) = \bigoplus\limits_{rk(b)\ge q} C_\bullet(b,q)$, with differential $\delta(q) = \oplus \delta_{b,q}$. 

Finally we
define $\hat C_\bullet = \bigoplus\limits_{q\ge1}C_\bullet(q)$ with chain differential $\hat \delta = \oplus\delta(q)$.  
\end{defn}

The next step is to parlay the chain complex $\hat C_\bullet$ into a chain complex of free left $A'$ modules.

\begin{defn}\label{complex}  For all $n\ge 0$ we define maps 
$$d:A'\otimes \hat C_n \to A' \otimes \hat C_{n-1}$$ 
by extending linearly from the following formula.  
Fix ${(b_n>\cdots>b_0)\over b,q} \in C_n(\Delta(\Gamma_{b,q}))$ for some $b\in \Gamma_+$ and $q\le rk(b)$.  
Let $k = d(b,b_n)$.  Then for $m\in A'$:
$$d(m\otimes {(b_n>\cdots> b_0)\over b,q}) = m \cdot f(b,b_n) \otimes {(b_{n-1}>\cdots>b_0)\over b_n,q-k} - 
		m\otimes \hat \delta({(b_n>\cdots >b_0)\over b,q}).$$
\end{defn}

\begin{lemma}  $(A'\otimes \hat C_\bullet,d)$ is a chain complex of free left  $A'$ modules.
\end{lemma}

\begin{proof}  Fix ${(b_n>\cdots>b_0)\over b,q} \in C_n(\Delta(\Gamma_{b,q}))$ and $m\in A'$.  We may assume $n\ge 1$.  Set 
$k=d(b,b_n)$ and $j = d(b_n,b_{n-1})$.  Then we have
\begin{align*}
d(m\cdot f(b,b_n)\otimes {(b_{n-1}>\cdots>b_0)\over b_n,q-k} ) 
	&=m\cdot f(b,b_n)f(b_n,b_{n-1})\otimes {(b_{n-2}>\cdots> b_0)\over b_{n-1},q-k-j}\\
	 &\qquad\qquad -m\cdot f(b,b_n) \otimes \hat \delta({(b_{n-1}>\cdots>b_0)\over b_n,q-k})\\
	 &=m\cdot f(b,b_{n-1}) \otimes {(b_{n-2}>\cdots> b_0)\over b_{n-1},q-k-j}\\
	&\qquad\qquad -m\cdot f(b,b_n) \otimes \hat \delta({(b_{n-1}>\cdots>b_0)\over b_n,q-k})\\
\end{align*}
and similarly
\begin{align*}
d(-m\otimes \hat \delta({(b_n>\cdots >b_0)\over b,q}\,))
	&=- d(m\otimes {(b_{n-1}>\cdots>b_0)\over b,q})\\	
	&\qquad   -d(m\otimes \sum\limits_{i=1}^n (-1)^{i}{(b_n>\cdots>b_{n-i+1}>b_{n-i-1}>\cdots >b_0)\over b,q})\\
	&= -m\cdot f(b,b_{n-1}) \otimes {(b_{n-2}>\cdots>b_0)\over b_{n-1},q-k-j}\\
	&\qquad + m\otimes \hat \delta({(b_{n-1}>\cdots>b_0)\over b,q})\\
	& -  m\cdot f(b,b_n) \otimes \sum\limits_{i=1}^n(-1)^{i}{(b_{n-1}>\cdots>b_{n-i+1}>b_{n-i-1}>\cdots >b_0)\over b_n,q-k}\\
	&+ m\otimes \hat\delta \left( \sum\limits_{i=1}^n (-1)^{i}{(b_n>\cdots>b_{n-i+1}>b_{n-i-1}>\cdots >b_0)\over b,q}\right)\\
	&= -m\cdot f(b,b_{n-1}) \otimes {(b_{n-2}>\cdots>b_0)\over b_{n-1},q-k-j}\\
	&\qquad\qquad + m\otimes \hat \delta^2({(b_{n}>\cdots>b_0)\over b,q})\\
	& +  m\cdot f(b,b_n) \otimes \hat\delta({(b_{n-1}>\cdots>b_0)\over b_n,q-k})\\
\end{align*}	
Since $\hat \delta^2 = 0$, the sum of the two expressions above is 0, showing that $d^2=0$ as required.		
\end{proof}

The complex $A'\otimes \hat C_\bullet$ can be augmented by defining $d: A'\otimes \hat C_{-1} \to A'$ via the formula:
$d(m\otimes {(\,)\over b,q} ) = m e'(b,q)$.   We recall that whenever $(x,j)\mydot (y,k)$, $f(x,y)e'(y,k) = e'(x,j+k)$, and hence
$d^2:A'\otimes \hat C_0 \to A$ is zero.

The essential key to the proof of \ref{intro1} is the following:

\begin{prop}\label{resolution}  
$A'\otimes \hat C_\bullet\to A'\to \F$ is a free resolution of the left $A'$-module $\F=A'/A'_+$.  
\end{prop}

\begin{proof}
We define a linear map $\zeta:A'\otimes \hat C_n \to A'\otimes \hat C_{n+1}$ for all $n\ge -1$ by extending linearly from the following formula.   Choose
a basis element $\beta = {(b_n>\cdots >b_0)\over b,q}$ in $C_n(\Delta(\Gamma_{b,q}))$. Let $k = d(b,b_n)$ if $n\ge 0$ and let $k=1$ if $n=-1$.  
Let
 $m$ be any good monomial in $\cM'_\Gamma$ and write 
$m = e'(z_1,j_1)e'(z_2,j_n)\cdots e'(z_r,j_r)$ with the usual conditions: $(z_i,j_i)\notmydot (z_{i+1},j_{i+1})$. 
Let $m' =e'(z_1,j_1)e'(z_2,j_n)\cdots e'(z_{r-1},j_{r-1})$.   Then 
$$\zeta(m\otimes {(b_n>\cdots >b_0)\over b,q}) = \begin{cases} 
		m'\otimes {(b>b_n>\cdots>b_0)\over z_r,q+j_r} & \hbox{if\ } r\ge 1\hbox{ and } (z_r,j_r)\mydot (b,k),\\ 
		0 & \hbox{otherwise}.\\ 
\end{cases}$$

Observe that in the first case of the formula, $f(z_r,b) = e'(z_r,j_r)$. Moreover, when $n\ge 0$, 
$e'(z_r,j_r+k) = e'(z_r,j_r)f(b,b_n)$.

Similarly we define $\zeta:A\to A\otimes \hat C_{-1}$ as follows.  For $m\in \cM'_\Gamma$, exactly as above, define $\zeta(m)$ by
$$\zeta(m) = \begin{cases} 
	m'\otimes {(\,)\over z_r,j_r} &\hbox{ if } r\ge 1,\\
	0 &\hbox{ if } m=1.\\
	\end{cases}
$$

\noindent {\bf Claim:} $z$ is a homotopy, that is: $\zeta d + d\zeta = 1$.  

Case 1: Given $\beta$ and $m$ as above, assume $n\ge 0$, $r\ge1$ and  $(z_r,k_r)\mydot (b,k)$. Using the observation above we have: 
\begin{align*}
d \zeta (m\otimes \beta) &= d(m'\otimes {(b>b_n>\cdots>b_0)\over z_r,q+j_r})\\
& = m\otimes \beta - m' \otimes \hat \delta({(b>b_n>\cdots>b_0)\over z_r,q+j_r}).\\
\end{align*}
We also have:
\begin{align*}
\zeta (m\otimes \hat \delta (\beta)) 
&=\zeta(m\otimes \sum\limits_{i=0}^{n}(-1)^i {(b_n>\cdots >b_{n-i+1}>b_{n-i-1}>\cdots>b_0)\over b,q}))\\
&= m'\otimes \sum\limits_{i=0}^n (-1)^i{(b>b_n>\cdots >b_{n-i+1}>b_{n-i-1}>\cdots>b_0)\over z_r,q+j_r} \\
&= m' \otimes {(b_n>\cdots>b_0)\over z_r,q+j_r} - m'\otimes \hat \delta ( {(b>b_n>\cdots>b_0)\over z_r,q+j_r})\\
\end{align*}
Therefore:
\begin{align*}
\zeta d(m\otimes \beta) & = 
   \zeta(m'e'(z_r,j_r)f(b,b_n) \otimes{(b_{n-1}>\cdots>b_0)\over b_n,q-k}) - \zeta(m\otimes \hat \delta(\beta)) \\
   &=  \zeta(m'e'(z_r,j_r+k) \otimes{(b_{n-1}>\cdots>b_0)\over b_n,q-k}) - \zeta(m\otimes \hat \delta(\beta))\\
   &= m' \otimes {(b_n>\cdots>b_0)\over z_r,q+j_r} - \zeta(m\otimes \hat \delta(\beta))\\
   &= m'\otimes \hat \delta ( {(b>b_n>\cdots>b_0)\over z_r,q+j_r})\\
\end{align*}
Thus $(\zeta d + d\zeta)(m\otimes \beta) = m\otimes \beta$. 

Case 2: Given $\beta$ and $m$ as above, assume $n=-1$, $r\ge1$ and $(z_r,k_r)\mydot (b,1)$.   Then
\begin{align*}
d\zeta(m\otimes \beta) &=d(m'\otimes {(b)\over z_r,q+j_r}) \\
&= m\otimes {(\,)\over b,q} - m'\otimes {(\,)\over z_r,q+j_r}\\
&= m\otimes \beta - m'\otimes {(\,)\over z_r,q+j_r}\\
\end{align*}
On the other hand,
$$
\zeta d(m\otimes \beta) = \zeta( m e'(b,q) ) = \zeta(m' e'(z_r,j_r+q) ) = m'\otimes {(\,)\over z_r,q_r+j}.$$
Hence $(\zeta d + d\zeta)(m\otimes \beta) = m\otimes \beta$.

Case 3:  Given $\beta$ and $m$ as above, assume $n\ge 0$ and $r\ge1$ but $d(z_r,b) \ne j_r$, i.e. $(z_r,k_r)\notmydot (b,k)$.  
Note that in this case, $m f(b,b_n) = m'e'(z_r,j_r)e'(b,k)$ is a good monomial.  Thus $d'\zeta (m\otimes \beta) = 0$, but
\begin{align*}
\zeta d(m\otimes \beta) & = \zeta(m'e'(z_r,j_r)f(b,b_n) \otimes {(b_{n-1}>\cdots>b_0)\over b_n,q-k}) - \zeta(m\otimes \hat\delta (\beta))\\
& = \zeta(m'e'(z_r,j_r)e'(b,k) \otimes {(b_{n-1}>\cdots>b_0)\over b_n,q-k}) - 0\\
&= m\otimes \beta \\
\end{align*}
Thus $(\zeta d + d\zeta)(m\otimes \beta) = m\otimes \beta$. 

Case 4:   Given $\beta$ and $m$ as above, assume $n=-1$ and $r\ge1$ but $d(z_r,b) \ne j_r$, i.e. $(z_r,k_r)\notmydot (b,1)$.
Then $d\zeta(m\otimes \beta) = 0$, whereas
$$\zeta d(m\otimes \beta) = \zeta( m e'(b,q)) = m\otimes {(\,)\over b,q} = m\otimes \beta.$$
Thus $(\zeta d + d\zeta)(m\otimes \beta) = m\otimes \beta$. 

Case 5:  Finally, let $\beta$ and $m$ be as above, but assume $r=0$, i.e. $m=1$.  Then $d\zeta (1\otimes \beta) = 0$ 
and exactly as in cases 3 and 4 above: $\zeta d(1\otimes \beta) = 1\otimes \beta$. 

This completes the proof of the claim.  It follows at once that the chain complex $A'\otimes \hat C_\bullet \to A'$ is acyclic.  But the 
image of $A'\otimes \hat C_{-1}$ in $A'$ is $\sum\limits_{b,q} A'e'(b,q) = A'_+$.  This proves the Proposition.
\end{proof}

\begin{rmk}  It is not difficult to embed the chain complex $A'\otimes \hat C_\bullet$ in the standard bar complex for $A'$. The reader may also have noticed that the formulas for $d$ and $\zeta$ on $A'\otimes \hat C_\bullet$ would work just as 
well with $A$ in place of $A'$ (once the definition of $f(b,a)$ has been modified to be the product of the edges of any path from $b$ to $a$).  However, it is easy to see that $coker(A\otimes \hat C_0 \buildrel d \over \to A\otimes \hat C_{-1})$ is
not isomorphic to $A_+$.   
\end{rmk}

We have all but proved the following version of \ref{intro1}.

\begin{thm} For any finite ranked poset $\Gamma$, $p\ge 1$ and $q\ge 0$:
$$\Tor^{A'}_{p,q}(\F,\F) \cong
\bigoplus\limits_{q\le rk(b)} \tilde H_{p-2}(\Delta(\Gamma_{b,q})).$$
\end{thm}

\begin{proof}  Since $A'\otimes \hat C_\bullet \to A'\to \F$ is a free resolution of $\F$, we see at once that for $p\ge 1$,
$\Tor^{A'}_{p}(\F,\F) \cong \tilde H_{p-2}(\hat C_\bullet)$.   The degree $q$ part of the right hand side of this equation is exactly
$\tilde H_{p-2}(C_\bullet(q))$.   This observation proves the theorem. 
\end{proof}

We have proved the second formula of \ref{intro1}. The first formula of \ref{intro1} follows by duality. 
 
\section{Proof of \ref{intro2} and some simple combinatorics}

There are two small pieces of combinatorics required to derive \ref{intro2} from \ref{intro1}. 
The first is the definition of uniform 
as it originated in \cite{RSW2}.  For $b\in \Gamma$, the set $s_b(k)$ is defined to be 
$\{ a\in \Gamma\, |\,d(b,a)= k\}$.

\begin{defn}\label{uniform} 
Let $\Gamma$ be a ranked poset.  For $x\in \Gamma$ and $a,b \in S_x(1)$, write
$a\sim_x b$ if there exists $c\in S_a(1) \cap S_b(1)$ and extend $\sim_x$ to an equivalence relation on 
$S_x(1)$.   We say that 
$\Gamma$ is {\it uniform} if, for every $x\in \Gamma$, $\sim_x$ has a unique equivalence class.  
\end{defn}

It is apparent that $\Gamma$ is uniform if and only if the subgraphs $\Gamma_{b,3}$ are connected as graphs 
for every $b$ of rank at 
least $3$ in $\Gamma$.  It is also clear that if $\Gamma_{b,n}$ is disconnected for some $n>3$, then 
$\Gamma_{b,3}$ is also disconnected.   
But the graph $\Gamma_{b,n}$ is connected  if and only if $\tilde H^0(\Delta(\Gamma_{b,n})) = 0$.  Thus we have
$$\Gamma \hbox{ is uniform if and only if } \bigoplus\limits_{rk(b) \ge q\ge 3} \tilde H^0(\Delta(\Gamma_{b,q})) = 0.$$
By \ref{intro1}, this is equivalent to $\Ext_{A_\Gamma}^{2,q}(\F,\F) = 0$ for all $q>2$, or equivalently that 
$A_\Gamma$ is a 
quadratic algebra.  This proves (1) of \ref{intro2}.

There are now two straightforward ways to prove (2) of \ref{intro2}.  We could cite Theorem 1.2 of \cite{KS} 
and combine it with \ref{intro1} and (1) of \ref{intro2} to get an immediate proof of (2) of \ref{intro2}. 
Another approach is to prove Proposition \ref{prop}, given below.  This then gives a completely new proof of 
Theorem 1.2 of \cite{KS}.

From \ref{intro1}, we see that the algebra $A'_\Gamma$ is Koszul if and only if $\Gamma$ satisfies the 
following combinatorial property:
\begin{center}
(*) \ \ \  For $b\in \Gamma_+$ and all triples $rk(b) \ge q > n$;  $\tilde H^{n-2}(\Delta(\Gamma_{b,q})) = 0$.\qquad\qquad
\end{center}
Recall that $\dim \Delta(\Gamma_{b,q}) = q-2$. 

\begin{prop}\label{prop}  Let $\Gamma$ be a finite ranked poset.  Then $\Gamma$ is Cohen-Macaulay if and only if 
$\Gamma$ satisfies condition {\rm (*)}.
\end{prop}

\begin{proof}
We start with the following observation from \cite{KS}.  Choose $b\in \Gamma_+$ and $p\le rk(b)$.  Then 
$\Delta(\Gamma_{b,p-1})$ is a closed subspace of $\Delta(\Gamma_{b,p})$.   The relative cohomology group 
$H^n(\Delta(\Gamma_{b,p}),\Delta(\Gamma_{b,p-1}))$  decomposes as 
$$H^n(\Delta(\Gamma_{b,p}),\Delta(\Gamma_{b,p-1})) = 
	\bigoplus\limits_{a\in S_b(p-1)} \tilde H^{n-1}(\Delta(\,(a,b)\,)).$$
We thus get a long exact sequences in cohomology:

\begin{align*}
\cdots \to \tilde H^n(\Delta(&\Gamma_{b,p})) \to  \tilde H^n(\Delta(\Gamma_{b,p-1})) \to
	\bigoplus\limits_{a\in S_b(p-1)}  \tilde H^{n}(\Delta(\,(a,b)\,))\\
	& \to  \tilde H^{n+1}(\Delta(\Gamma_{b,p})) \to \tilde H^{n+1}(\Delta(\Gamma_{b,p-1})) \to\cdots\\
\end{align*}

Suppose now that $\Gamma$ is Cohen-Macaulay, but fails condition (*).  Choose a triple $rk(b)\ge q>n$ for which
$\tilde H^{n-2}(\Delta(\Gamma_{b,q})) \ne 0$.  We may assume $q$ is minimal amongst all such examples.  But then
either $\tilde H^{n-2}(\Delta(\Gamma_{b,q-1})) = 0$ or $n=q-1$.  In the former case, by the long exact sequence, 
$\tilde H^{n-3}(\Delta(\,(a,b)\,)) \ne 0$ for some $a\in S_b(q-1)$.  Since $dim(\Delta(\,(a,b)\,) = q-3 >n-3$,  
this contradicts 
Cohen-Macaulay. Hence we must have $n=q-1$. 

Now we note that $rk(b)$ must be strictly larger than $q$, since otherwise $\Gamma_{b,q} = (*,b)$ and we would have
$\tilde H^{q-3}(\Delta(\,(*,b)\,)) \ne 0$, contradicting Cohen-Macaulay again.  Hence we have 
$\Delta(\Gamma_{b,q})$ as a 
proper closed subspace of $\Delta(\Gamma_{b,q+1})$.  Using Cohen-Macaulay 
again, the long exact sequence implies 
$\tilde H^{q-3}(\Delta(\Gamma_{b,q+1})) \ne 0$.  We can repeat this argument again and again, to get
$\tilde H^{q-3}(\Delta(\Gamma_{b,q+j})) \ne 0$ for all $j\ge 0$.  But for sufficiently large $j$, 
$\Gamma_{b,q+j} = (*,b)$, and we have arrived at the same contradiction.

Conversely, assume condition (*) holds for $\Gamma$.  Then in the long exact sequence, any summand
 $\tilde H^n(\Delta(\,(a,b)\,))$ is trapped between two terms that must be zero whenever 
 $n< \dim(\Delta(\,(a,b)\,))$.  Hence 
 $\Gamma$ is Cohen-Macaulay.  
\end{proof}

The proof of (2) of \ref{intro2} is now apparent.

\section{Proof of \ref{intro3}}

We prove \ref{intro3} by one very simple example.  Let $\Gamma$ be the finite ranked poset in the figure below.  Throughout this section let $A= A_\Gamma$ and $A'= A'_\Gamma$.

\begin{center}
\includegraphics[width=.85in]{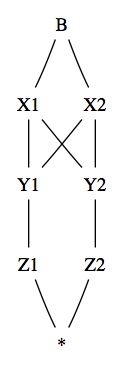}
\end{center}

\begin{lemma}\label{cex1}  The dimensions of the non-zero spaces $\Ext^{p,q}_{A'}(\F,\F)$ are as follows:
\begin{align*} 
\dim(\Ext^{1,1}_{A'}(\F,\F))&=7\\
\dim(\Ext^{2,2}_{A'}(\F,\F))&=3\\
\dim(\Ext^{2,3}_{A'}(\F,\F))&=2\\
\dim(\Ext^{3,3}_{A'}(\F,\F))&=1\\
\dim(\Ext^{3,4}_{A'}(\F,\F))&=1\\
\end{align*}
\end{lemma}

\begin{proof}
This follows immediately from \ref{intro1} by analyzing the spaces $\Gamma_{v,q}$ for various $v\in \Gamma_+$ and $q\le rk(v)$.  
We observe that $\Gamma_{B,2}$, $\Gamma_{X1,2}$ and $\Gamma_{X2,2}$ are contributing the three dimensions to 
$\Ext^{2,2}_{A'}(\F,\F)$.  Similarly $\Gamma_{X1,3}$ and $\Gamma_{X2,3}$ are each contributing one dimension to 
$\Ext^{2,3}_{A'}(\F,\F)$ and 0 to $\Ext^{3,3}_{A'}(\F,\F)$, whereas $\Gamma_{B,3}$ contributes 0 to $\Ext^{2,3}_{A'}(\F,\F)$ and dimension $1$ to $\Ext^{3,3}_{A'}(\F,\F)$ (since $\Delta(\Gamma_{B,3})$ is $S^1$.).  Finally, $\Gamma_{B,4}$ is contributing 
only to $\Ext^{3,4}_{A'}(\F,\F)$ since $\Delta(\Gamma_{B,4}) = \Delta(\,(*,B)\,)$ is a two-dimensional space homotopic to $S^1$.    
\end{proof}

As a consequence of \ref{cex1}, we see that $A'$, which is generated by seven linear elements, has a minimal set of relations consisting of 3 quadratic relations and 2 cubic relations.  Now we turn to $A$.

\begin{lemma}  The algebra $A$ is generated by seven linear elements and has a minimal set of relation consisting of 3 quadratic relations and at most one cubic relation.
\end{lemma}

\begin{proof}  We proceed by brute force.

After the linear change of variables described just after the definition \ref{original}, $A$ has seven linear generators $v_x$, $x\in \Gamma_+$.   We will drop the $v_x$ notation and simply write $x$.  Let $V$ be the $\F$-space on the seven generators: $B, X1,X2,Y1,Y2,Z1$ and $Z2$.

There are two paths in $\Gamma$ from $B$ to $Y1$ and two paths from $B$ to $Y2$, each giving us one quadratic relation.  There are two paths from $X1$ to $*$ and two from $X2$ to $*$ giving us two more quadratic relations.  There are also multiple paths from $B$ to $*$, but it is clear that these give us relations that are already encoded in the previous four.  Those four relations, written as 2-tensors are
\begin{align*}
Q_1 &=(B-X1)\otimes (X1-Y1) - (B-X2)\otimes(X2-Y1)\\
Q_2&=(B-X2)\otimes (X2-Y2) - (B-X1)\otimes(X1-Y2)\\
Q_3&=(X1-Y1)\otimes (Y1-Z1) + (X1-Y1)\otimes Z1+(Y1-Z1)\otimes Z1\\
	&\qquad-(X1-Y2)\otimes (Y2-Z2)-(X1-Y2)\otimes Z2-(Y2-Z2)\otimes Z2\\
Q4&=(X2-Y1)\otimes (Y1-Z1) + (X2-Y1)\otimes Z1+(Y1-Z1)\otimes Z1\\
	&\qquad-(X2-Y2)\otimes (Y2-Z2)-(X2-Y2)\otimes Z2-(Y2-Z2)\otimes Z2\\
\end{align*}
By direct inspection we see:
$$Q_1+Q_2 = Q_3 - Q_4 = (X1-X2)\otimes (Y1-Y2)$$
and thus the space of quadratic relations has dimension at most 3.  By inspection of monomials, it is clear that the span of 
the four quadratic relations is exactly 3.

Each of the four paths from $B$ to $*$ gives a cubic element $e(\pi,3)$.  These elements are
\begin{align*}
C_1 &=(B-X1)\otimes(X_1-Y_1)\otimes(Y1-Z1) +(B-X1)\otimes(X1-Y1)\otimes Z1\\
		&\qquad\qquad +(B-X1)\otimes(Y1-Z1)\otimes Z1 + (X1-Y1)\otimes(Y1-Z1)\otimes Z1\\
C_2 &=(B-X1)\otimes(X_1-Y_2)\otimes(Y2-Z2) +(B-X1)\otimes(X1-Y2)\otimes Z2\\
		&\qquad\qquad +(B-X1)\otimes(Y2-Z2)\otimes Z2 + (X1-Y2)\otimes(Y2-Z2)\otimes Z2\\
C_3 &=(B-X2)\otimes(X2-Y_2)\otimes(Y2-Z2) +(B-X2)\otimes(X2-Y2)\otimes Z2\\
		&\qquad\qquad +(B-X2)\otimes(Y2-Z2)\otimes Z2 + (X2-Y2)\otimes(Y2-Z2)\otimes Z2\\
C_4 &=(B-X2)\otimes(X2-Y_1)\otimes(Y1-Z1) +(B-X2)\otimes(X2-Y1)\otimes Z1\\
		&\qquad\qquad +(B-X2)\otimes(Y1-Z1)\otimes Z1 + (X2-Y1)\otimes(Y1-Z1)\otimes Z1\\						
\end{align*}
The two paths between $X1$ and $*$ also generate cubic elements $e(\pi,3)$:
\begin{align*}
U_1 &=(X1-Y1)\otimes(Y_1-Z_1)\otimes Z1\\
V_1 &=(X1-Y2)\otimes(Y_2-Z_2)\otimes Z2\\				
\end{align*}
And finally the two paths between $X2$ and $*$ generate two more cubic elements $e(\pi,3)$:
\begin{align*}
U_2 &=(X2-Y1)\otimes(Y_1-Z_1)\otimes Z1\\
V_2 &=(X2-Y2)\otimes(Y_2-Z_2)\otimes Z2\\				
\end{align*}

The cubic relations given by the equations of the form $e(\pi,3) = e(\pi',3)$, from \ref{original}, are then the following:
$C1-C2$, $C2-C3$, $C3-C4$, $U1-V1$ and $U2-V2$.  We make the following direct observations:
\begin{align*}
C_1-C_2 &= (B-X1)\otimes Q_3 + (U_1-V_1)\\
C_2-C_3 &= -Q_2\otimes Y2\\
C_3-C_4 &= (B-X2)\otimes (-Q_4) + (V_2-U_2)\\
C_1-C_4 &= Q_2\otimes Y1\\
\end{align*}
The first three of these equations tell us that $(U_1-V_1)$ and $(U_2-V_2)$ generate the other three cubic relations (modulo quadratic relations).  But adding the first three and equating to the the fourth yields the new equation:
$$ (B-X1)\otimes Q_3 - Q_2 \otimes Y2 - (B-X2)\otimes Q_4 - Q_2\otimes Y_1 + (U_1-V_1) = (U_2-V_2)$$
This proves that either $(U_2-V_2)$ or $(U_1-V_1)$ alone, together with the quadratic relations, generates all of the other cubic relations. This proves the lemma.  
\end{proof}

\begin{rmk}  We know that the algebras $A$ and $A'$ have the same Hilbert series, which can be computed directly from the information in Lemma \ref{cex1}.  We have shown, however that $\dim\Ext_A^{2,3}(\F,\F)$ is at most 1.  By a simple Hilbert series argument, we can then calculate the following dimensions:
\begin{align*} 
\dim(\Ext^{1,1}_{A}(\F,\F))&=7\\
\dim(\Ext^{2,2}_{A}(\F,\F))&=3\\
\dim(\Ext^{2,3}_{A}(\F,\F))&=1\\
\dim(\Ext^{3,3}_{A}(\F,\F))&=0\\
\dim(\Ext^{3,4}_{A}(\F,\F))&=1\\
\end{align*}
\end{rmk}

Lemma \ref{cex1} and the remark above combine to prove Theorem \ref{intro3}.

 \bibliographystyle{amsplain}
\bibliography{bibliog}
\end{document}